\documentclass[12pt]{amsart}
\usepackage{mathrsfs}
\usepackage{txfonts}
\usepackage{amssymb}
\usepackage{amsmath,amssymb}
\newtheorem{theorem}{Theorem}[section]
\newtheorem{proposition}[theorem]{Proposition}

\newtheorem{corollary}[theorem]{Corollary}
\theoremstyle{definition}

\theoremstyle{remark}

\numberwithin{equation}{section}
\allowdisplaybreaks

\allowdisplaybreaks

\def\rr{{\mathbb R}}
\def\rn{{{\rr}^n}}

\def\cm{{\mathrm M}}

\def\fz{\infty}
\def\az{\alpha}

\def\dz{\delta}

\def\hs{\hspace{0.3cm}}

\def\ls{\lesssim}

\def\dlim{\displaystyle\lim}

\def\r{\right}
\def\lf{\left}

\begin{document}

\arraycolsep=1pt
\title[Regularity and Capacity for the Fractional Dissipative Operator]
{Regularity and Capacity for the Fractional Dissipative Operator}


\author{Renjin Jiang}
\address{School of Mathematical Sciences, Beijing Normal University, Laboratory of Mathematics
and Complex Systems, Ministry of Education, Beijing, 100875, China}
\email{rejiang@bnu.edu.cn}
\author{Jie Xiao}
\address{Department of Mathematics and Statistics, Memorial University of Newfoundland, St. John's, NL A1C 5S7, Canada}
\email{jxiao@mun.ca}
\author{Dachun Yang}
\address{School of Mathematical Sciences, Beijing Normal University, Laboratory of Mathematics
and Complex Systems, Ministry of Education, Beijing, 100875, China}
\email{dcyang@bnu.edu.cn}
\author{Zhichun Zhai}
\address{Department of Mathematical and Statistical Sciences, University of Alberta, Edmondon, AB T6G 2G1, Canada}
\email{zhichun1@ualberta.ca}

\thanks{Jie Xiao was in part supported by NSERC of Canada and URP of Memorial
University, Canada. Dachun Yang was supported by the National Natural Science Foundation (Grant No. 11171027) of
China and Program for Changjiang Scholars and Innovative Research Team in University of China.}

\date{}


\begin{abstract} This note is devoted to exploring some analytic-geometric properties of the
regularity and capacity associated to the so-called fractional dissipative
operator $\partial_t+(-\Delta)^\alpha$, naturally establishing a diagonally sharp Hausdorff dimension estimate for the blow-up set of a weak solution to the fractional dissipative equation $(\partial_t+(-\Delta)^\alpha)u(t,x)=F(t,x)$ subject to $u(0,x)=0$.
\end{abstract}

\maketitle

\tableofcontents \pagenumbering{arabic}


\section{Introduction and the main results}\label{s1}

This beginning part is designed to describe the principal results of this article.

\subsection{The fractional dissipative equation} For $n=1,2,3,...$, $\alpha\in (0,1]$
and $\mathbb R_+:=(0,\infty)$, let $\mathbb{R}^{1+n}_{+}:=\mathbb R_+\times\mathbb{R}^{n}$
be the upper half space of the $1+n$ dimensional Euclidean space $\mathbb R^{1+n}$
and $(-\Delta)^{\alpha}$ be the fractional Laplace operator which is determined by
$$
(-\Delta)^{\alpha}u(\cdot,x):
   =\mathcal {F}^{-1}(|\xi|^{2\alpha}\mathcal{F}u(\cdot,\xi))(x),\ \forall x\in\rn,
$$
where  $\mathcal {F}$ denotes  the Fourier transform and  $\mathcal {F}^{-1}$ its
inverse:
$$
    \begin{cases}
    \mathcal{F}(g)(x):=(2\pi)^{-n/2}\int_{\mathbb R^n}e^{-ix\cdot y}g(y)\,dy;\\
     \mathcal{F}^{-1}(g)(x):=(2\pi)^{-n/2}\int_{\mathbb R^n}e^{ix\cdot y}g(y)\,dy.
    \end{cases}
$$
     From the celebrated Duhamel's principle it follows that a weak solution $u(t,x)$
     to the fractional dissipative equation living in fluid dynamics via the so-called fractional dissipative operator
     $L^{(\alpha)}:=\partial_{t}+(-\Delta)^{\alpha}$:
$$
\begin{cases}\label{1a}
L^{(\alpha)}u(t,x)=F(t,x),\quad\forall (t,x)\in\mathbb{R}^{1+n}_{+};\\
u(0,x)=f(x),\quad\forall x\in\mathbb{R}^n,
\end{cases}
$$
namely (cf. \cite{NSS, FJR}),
$$
\begin{cases}
\iint_{\mathbb R^{1+n}_+} u\tilde{L}^{(\alpha)}\phi\,dxdt\!=\!-\iint_{\mathbb R^{1+n}_+}
F\phi\,dxdt-\!\int_{\mathbb R^n}f(x)\phi(0,x)\,dx,\, \forall \phi\in C_0^\infty(\mathbb R^{1+n}_+)\\
\hbox{with}\\
\tilde{L}^{(\alpha)}\phi(t,x)=-\partial_t\phi(t,x)+\Big(\frac{(1-\alpha)2^{2\alpha}
\Gamma\big(\frac{n+2\alpha}{2}\big)}{\pi^{n/2}\Gamma(1-\alpha)}\Big)
{\dlim_{\epsilon\to 0}}\int_{\{y\in\mathbb R^n: |y|>\epsilon\}}\frac{\phi(t,x+y)-\phi(t,x)}{|y|^{-n-2\alpha}}\,dy,
\end{cases}
$$
can be written as
$$
u(t,x)=R_\alpha f(t,x)+S_\alpha F(t,x),
$$
where
$$
\begin{cases}
R_\alpha f(t,x):=e^{-t(-\Delta)^\alpha}f(x);\\
S_\alpha F(t,x):=\int_{0}^{t}e^{-(t-s)(-\Delta)^{\alpha}}F(s,x)\,ds,
\end{cases}
$$
for which
$$
\begin{cases}
e^{-t(-\Delta)^{\alpha}}v(\cdot,x):=K_{t}^{(\alpha)}(x)\ast v(\cdot,x);\\
   K_{t}^{(\alpha)}(x):=(2\pi)^{-n/2}\int_{\mathbb R^n}e^{ix\cdot y-t|y|^{2\alpha}}\,dy,
   \end{cases}
$$
   and $\ast$ represents the convolution operating on the space variable. Here it is perhaps appropriate to mention that
   $$
   \begin{cases}
   K_{t}^{(1)}(x)=(4\pi)^{-n/2}e^{-|x|^2/(4t)}\\
   \hbox{and}\\
   K_t^{(1/2)}(x)=\pi^{-(1+n)/2}\Gamma\big((n+1)/2\big)t(t^2+|x|^2)^{-(1+n)/2}
   \end{cases}
   $$
   are the heat and Poisson kernels, respectively. Of course, $\Gamma(\cdot)$ is the classical gamma function. Although an explicit formula of $K_t^{(\alpha)}(x)$ for $\alpha\in (0,1]\setminus\{1/2,1\}$ is unknown (cf. \cite{MiaoYuanZhang, CW, Wu2, Wu Yuan, Chen4} and \cite{Zhai1, NSS, NSY1, NSY2, NSY3} for some related information), one has the following basic estimate (cf. \cite{XieZhang, ChenSong}): under $\alpha\in (0,1)$
   $$
    K_t^{(\alpha)}(x)\approx t\big(t^{\frac1{2\alpha}}+|x|\big)^{-(n+2\alpha)},\quad\forall (t,x)\in \mathbb R_+^{1+n}.
    $$
In the above and below,  $\mathsf{X}\approx\mathsf{Y}$ means $\mathsf{Y}\lesssim\mathsf{X}\lesssim\mathsf{Y}$ where the second estimate means that there is a positive constant $c$, independent of main parameters,
such that $\mathsf{X}\le c\mathsf{Y}$. From now on, $\alpha$ will be always assumed to be in the interval $(0,1)$.

\subsection{Regularity for the fractional dissipative operator} The following function space regularity results of Strichartz type, plus \cite{Ad}, actually induce the research objective of this current paper.

\begin{theorem}\label{t11}
\item{\rm(i)} \cite[Lemma 3.2]{MiaoYuanZhang} If
$$
\begin{cases}
1\le p\le \tilde{p}<\frac{np}{n-\min\{n,2\alpha\}};\\
\frac{1}{\tilde{q}}=\big(\frac{n}{2\alpha}\big)\big(\frac{1}{p}-\frac{1}{\tilde{p}}\big),
\end{cases}
$$
then
$$
\|R_\alpha f\|_{L_t^{\tilde{q}}L_x^{\tilde{p}}(\mathbb R^{1+n}_+)}\lesssim\|f\|_{L^p(\mathbb R^n)}.
$$
\item{\rm(ii)} \cite[Theorem 1.4]{Zhai} If
$$
\begin{cases}
1\le p<\tilde{p}\le\infty;\\
1<q<\tilde{q}<\infty;\\
\big(\frac{1}{q}-\frac{1}{\tilde{q}}\big)+\big(\frac{n}{2\alpha}\big)\big(\frac{1}{p}-\frac{1}{\tilde{p}}\big)=1,
\end{cases}
$$
then
$$
\|S_\alpha F\|_{L_t^{\tilde{q}}L_x^{\tilde{p}}(\mathbb R^{1+n}_+)}\lesssim\|F\|_{L_t^{q}L_x^{p}(\mathbb R^{1+n}_+)}.
$$
\end{theorem}

Here and henceforth: $L^p_x(\mathbb R^n)$ denotes the usual Lebesgue $1\le p\le\infty$-space with respect to the space variable $x$; $L_{t}^{p_{2}}L_{x}^{p_{1}}(\mathbb{R}^{1+n}_+)$ is the mixed $(1\le p_1,p_2<\infty)$-Lebesgue space of all functions $F$ on $\mathbb R^{1+n}_+$ with
$$
\|F\|_{L_{t}^{p_{2}}L_{x}^{p_{1}}(\mathbb{R}^{1+n}_+)}
:=\left(\int_{\mathbb R_+}\left[\int_{\mathbb{R}^{n}}|F(t,x)|^{p_{1}}dx\right]^{\frac{p_{2}}{p_{1}}}
\,dt\right)^{\frac{1}{p_{2}}}<\infty,
$$
where a suitable modification is needed whenever $p_1$ or $p_2$ is $\infty$; for $\mathbb X=\mathbb R^n$ or $\mathbb R^{1+n}_+$ the symbols $C^\infty(\mathbb X)$, $C^\infty_0(\mathbb X)$ and $C(\mathbb X)$ stand for all infinitely smooth functions in $\mathbb X$, all infinitely smooth functions with compact support in $\mathbb X$ and all continuous functions in $\mathbb X$, respectively.

Throughout the paper, for each $(t_0,x_0)\in \rr^{1+n}_+$ and $r>0$, the parabolic ball is defined as
$$B^{(\alpha)}_{r}(t_{0},x_{0}):=\{(t,x)\in\mathbb{R}^{1+n}_{+}:\
|t-t_{0}|<r^{2\alpha}\ \&\ |x-x_0|<r\}
$$
and its volume is denoted by $|B_{r_0}^{(\alpha)}{(t_0,x_0)}|\approx r_0^{n+2\alpha}$.

The first main result of this paper appears as an essential extension or complement of Theorem \ref{t11}.

\begin{theorem}\label{t1new}

\item{\rm (i)} If $p\in[1,\infty]$ and $f\in L^p(\rr^n)$, then $R_\alpha f$ is continuous on $\rr^{1+n}_+$.

\item{\rm (ii)} If
$$
\begin{cases}
p\in[1,\infty);\\
1<q<\infty; \\
\frac{n}{p}+\frac{2\alpha}{q}=2\alpha;\\
(t_0,x_0)\in \rr^{1+n}_+;\\
r_0=t_0^{\frac {1}{2\az}};\\
0<\|F\|_{L^q_tL^p_x(\rr^{1+n}_+)}<\infty,
\end{cases}
$$
then there exists $C>0$ such that
$$\frac{1}{|B_{r_0}^{(\alpha)}{(t_0,x_0)}|}\iint_{B_{r_0}^{(\alpha)}{(t_0,x_0)}}\exp\lf(\frac{S_\az F(t,x)}
   {C\|F\|_{L^q_tL^p_x(\rr^{1+n}_+)}}\r)^{\frac{q}{q-1}}\,dx\,dt\ls 1.
$$

\item{\rm(iii)} If
$$
\begin{cases}
p\in[1,\infty);\\
1<q<\infty; \\
\frac{n}{p}+\frac{2\alpha}{q}<2\alpha;\\
(t,x)\in \rr^{1+n}_+;\\
\|F\|_{L^q_tL^p_x(\rr^{1+n}_+)}<\infty,
\end{cases}
$$
then $S_\alpha F$ is H\"older continuous in the sense that
\begin{eqnarray*}
   |S_\az F(t,x)-S_\az F(t_0,x_0)|&&\lesssim\big(|t-t_0|^{\frac{
   2\alpha-\frac{n}{p}-\frac{2\alpha}{q}}{2\alpha}}+|x-x_0|^{2\alpha-\frac{n}{p}-\frac{2\alpha}{q}}\big)
   \|F\|_{L^q_tL^p_x(\rr^{1+n}_+)}
\end{eqnarray*}
holds for any two sufficient close points $(t_0,x_0), (t,x)\in\rr^{1+n}_+$
\end{theorem}

\subsection{Capacity for the fractional dissipative operator} From Theorems \ref{t11}-\ref{t1new}
we know that it is necessary to estimate the size of the blow-up set of the so-called fractional dissipative potential $S_\alpha F$ below:
$$
\mathcal{B}[S_\alpha F;p,q]:=\{(t,x)\in \mathbb{R}^{1+n}_{+}: S_\alpha F(t,x)=\infty\}\quad \hbox{for}\quad 0\le F\in L_{t}^{q}L_{x}^{p}(\mathbb{R}^{1+n}_+).
$$

To handle this issue, let us introduce a new type of capacity. For a compact subset $K$ of $\mathbb{R}^{1+n}_{+},$ let
$$
C_{p,q}^{(\alpha)}(K)
:=\inf\Big\{\|F\|_{L^{q}_{t}L_{x}^{p}(\mathbb{R}^{1+n}_{+})}^{p\wedge
q}:\quad F\ge 0\ \&\ S_\alpha F\geq 1_{K}\Big\}
$$
be the $(\alpha,p,q)$-capacity of $K$ for the fractional dissipative
operator $L^{(\alpha)}$, where $1_{K}$ is the characteristic function of $K,$  $p\wedge q:=\min\{p,q\}$,
and $1\leq p,q<\infty.$ Moreover, the definition of $C_{p,q}^{(\alpha)}$ extends
to any subset of $\mathbb{R}^{1+n}_{+}$ in a similar way as \cite[Definitions 2.2.2 \& 2.2.4]{AH}.

Next, for
$$
\begin{cases}
0<\varepsilon\le\infty;\\
0<d<\infty;\\
K\subset\mathbb R^{1+n}_+;\\
B_{r_j}^{(\alpha)}(t_j,x_j):=\big\{(s,y)\in \mathbb{R}^{1+n}_{+}:\ |s-t_j|<r_j^{2\alpha}\ \&\ |y-x_j|<r_j\big\};\\
(t_j,x_j,r_j)\in\mathbb R_+\times\mathbb R^n\times\mathbb R_+;\\
\phi: [0,\infty)\mapsto [0,\infty]\ - \hbox{an\ increasing\ function\ with}\ \phi(0)=0,
\end{cases}
$$
let
$$
H_{\varepsilon}^{\phi,\alpha}(K):=\inf\left\{\sum_{j=1}^\infty \phi(r_{j}):\
K\subseteq\bigcup_{j=1}^\infty B^{(\alpha)}_{r_{j}}(t_{j},x_{j});\quad\hbox{with}\quad r_{j}\in (0,\varepsilon)\right\}
$$
be the $L^{\alpha}$-based $(\phi,\varepsilon)$-Hausdorff capacity of $K$. Then the $L^{\alpha}$-based $\phi$-Hausdorff measure of $K$ is defined by
$$
H^{\phi,\alpha}(K):=\lim_{\varepsilon\to
0}H^{\phi,\alpha}_{\varepsilon}(K).
$$
If $\phi(r):=r^d$ for all $r\in (0,\fz)$, then
$$
\begin{cases}
H_{\varepsilon}^{\phi,\alpha}(K)\equiv H_{\varepsilon}^{d,\alpha}(K);\\
H^{\phi,\alpha}(K)\equiv H^{d,\alpha}(K);\\
\hbox{dim}^{(\alpha)}_H(K):=\inf\{d: H^{d,\alpha}(K)=0\},
\end{cases}
$$
where the last quantity is called the $L^{(\alpha)}$-based Hausdorff dimension of $K$.

Below is our second theorem.

\begin{theorem}\label{t12new}

\item{\rm(i)} If
$$
\begin{cases}
1\le p<\infty;\\
1<q<\infty;\\
\frac{n}{p}+\frac{2\alpha}{q}-2\alpha>0,
\end{cases}
$$
then
$$
C_{p,q}^{(\alpha)}\big(B^{(\alpha)}_{r_0}(t_0,x_0)\big)\approx r_0^{(p\wedge q)\big(\frac{n}{p}+\frac{2\alpha}{q}-2\alpha\big)}\quad\hbox{as}\quad r_0\to 0\ \ \&\ \ (r_0,x_0)\in\mathbb R^{1+n}_+.
$$

\item{\rm(ii)} If
$$
\begin{cases}
1\le p<\infty;\\
1<q<\infty;\\
\frac{n}{p}+\frac{2\alpha}{q}-2\alpha=0,
\end{cases}
$$
then
$$
{C}_{p,q}^{(\alpha)}\big(B_r^{(\alpha)}(t_0,x_0)\big)\approx
\lf(\ln \frac 1{r_0}\r)^{(p\wedge q)(\frac{1}{q}-1)}\quad\hbox{as}\quad r_0\to 0\ \ \&\ \ (r_0,x_0)\in\mathbb R^{1+n}_+.
$$
\end{theorem}

As an immediate consequence of Theorems \ref{t11}-\ref{t1new}-\ref{t12new}, we get not only three geometric inequalities linking two types of capacity, but also some Hausdorff dimension estimates for the blow-up sets which are sharp in the diagonal case $p=q$.

\begin{corollary}\label{c31}
\item{\rm(i)} Let $\mathscr{L}^{1}(A)$ and $\mathscr{L}^{n}(B)$ stand for the $1$-dimensional and $n$-dimensional Lebesgue measures of bounded Borel sets $A\subset\mathbb R_+$ and $B\subset\mathbb{R}^{n}$, respectively. If
$$
\begin{cases}
1\le p<\tilde{p}<\infty;\\
1<q<\tilde{q}<\infty;\\
\beta:=(p\wedge q)\big(\frac{n}{p}+\frac{2\alpha}{q}-2\alpha\big)>0;\\
(\frac1{q}-\frac1{\tilde{q}})+(\frac{n}{2\alpha})(\frac1{p}-\frac1{\tilde{p}})=1,
\end{cases}
$$
then there is a $\delta\in (0,1)$ such that
$$
\big(\mathscr{L}^{1}(A)\big)^{\frac{p\wedge q}{\tilde{q}}}\big(\mathscr{L}^{n}(B)\big)^{\frac{p\wedge q}{\tilde{p}}}
\lesssim C^{(\alpha)}_{p,q}(A\times B)\lesssim H^{\beta,\alpha}_\delta(A\times B).
$$

\item{\rm(ii)} Let $K$ be a compact subset of $\mathbb R^{1+n}_+$. If
$$
\begin{cases}
1\le p<\infty;\\
1<q<\infty;\\
\beta:=(p\wedge q)\big(\frac{n}{p}+\frac{2\alpha}{q}-2\alpha\big)>0,
\end{cases}
$$
then  there is a $\delta\in (0,1)$ such that
$$
C_{p,q}^{(\alpha)}(K)\lesssim H_\delta^{\beta,\alpha}(K),
$$
and hence
$$
\hbox{dim}^{(\alpha)}_H(\mathcal{B}[S_\alpha F;p,q])\le n-2\alpha(p\wedge q-1)\ \ \mbox{provided}\ \ n-2\alpha(p\wedge q-1)>0.
$$

\item{\rm(iii)} Let $K$ be a compact subset of $\mathbb R^{1+n}_+$. If
$$
\begin{cases}
1\le p<\infty;\\
1<q<\infty;\\
\frac{n}{p}+\frac{2\alpha}{q}-2\alpha=0;\\
\phi(r):=(\ln_+ \frac1r)^{-(p\wedge q)(1-\frac1q)}, \ \forall r\in\mathbb R_+;\\
\ln_+ t:=\max\{0,\ln t\}, \ \forall t\in\mathbb R_+,
\end{cases}
$$
then there is a $\delta\in (0,1)$ such that
$$
C_{p,q}^{(\alpha)}(K)\lesssim H_\delta^{\phi}(K),
$$
and hence
$$
H^{\phi_\epsilon,\alpha}(\mathcal{B}[S_\alpha F;p,q])=0\ \ \hbox{provided}\ \
\begin{cases}
n-2\alpha(p\wedge q-1)=0;\\
\phi_\epsilon(r):=(\ln_+ \frac1r)^{-p\wedge q-\epsilon},\ \forall r\in\mathbb R_+;\\
\epsilon>0.
\end{cases}
$$
\end{corollary}

\section{Basics of the $(\alpha,p,q)$-capacity}\label{s2}

In order to demonstrate Theorems \ref{t1new}-\ref{t12new} and Corollary \ref{c31}, we need to know some basic facts on the $(\alpha,p,q)$-capacity.

\subsection{Duality of the $(\alpha,p,q)$-capacity} To establish the adjoint formulation of $C^{(\alpha)}_{p,q}$, we need to find out adjoint operator $S_\alpha^{\ast}$ of $S_\alpha$. Note that for any $F,G\in C_{0}^{\infty}(\mathbb{R}^{1+n}_+)$ one has
$$
\iint_{\mathbb R^{1+n}_+}S_\alpha F(t,x) G(t,x)\,dxdt=\int_{\mathbb R^{1+n}_+}F(t,x)\left(\int_{t}^{\infty}e^{-(s-t)(-\Delta)^{\alpha}}G(s,x)ds\right) dxdt.
$$
Thus, $S_\alpha^{\ast}G$ is given by setting, for all $(t,x)\in \mathbb R^{1+n}_+$,
$$
\big(S_\alpha^{\ast}G\big)(t,x):=\int_{t}^{\infty}e^{-(s-t)(-\Delta)^{\alpha}}
G(s,x)\,ds, \ \forall G\in C^\infty_0(\mathbb R^{1+n}_+).
$$
The definition of $S_\alpha^{\ast}$ can be extended to the family of Borel measures $\mu$ with compact support in $\mathbb R^{1+n}_+$. In fact, note that if $F$ is continuous and has a compact support in $\mathbb R^{1+n}_+$ and $\|\mu\|_1$ stands for the total variation of $\mu$ then a simple calculation with the equivalent estimate
$$
K^{(\alpha)}_t(x)\approx t(t^{1/(2\alpha)}+|x|)^{-n-2\alpha},\quad\forall (t,x)\in\mathbb R^{1+n}_+,
$$
gives
$$
\left|\iint_{\mathbb R^{1+n}_+}S_\alpha F\,d\mu\right|\lesssim\|\mu\|_1\sup_{(t,x)\in\mathbb R^{1+n}_+}|F(t,x)|.
$$
Hence an application of the Riesz representation theorem yields a Borel measure $\nu$ on $\mathbb R_+^{1+n}$ such that
$$
\iint_{\mathbb R^{1+n}_+}S_\alpha F\,d\mu=\int_{\mathbb R_+^{1+n}}F\,d\nu.
$$
This indicates that $S_\alpha^\ast\mu$ may be defined by $\nu$.

The above analysis leads to a dual description of the $(\alpha,p,q)$-capacity.

\begin{proposition}\label{p22} For a compact subset $K$ of $\mathbb R^{1+n}_+$ let $\mathcal{M_+}(K)$ be the class of all positive measures on $\mathbb R^{1+n}_+$ supported by $K$. If
$$
\begin{cases}
1<p,q<\infty;\\
p'=p/(p-1);\\
q'=q/(q-1),
\end{cases}
$$
then
$$
C_{p,q}^{(\alpha)}(K)=\sup\big\{\|\mu\|_1^{p\wedge q}:\ \mu\in\mathcal{M_+}(K)
\ \ \&\ \ \|S_\alpha^\ast\mu\|_{L_t^{q'}L^{p'}_x(\mathbb R^{1+n}_+)}\le 1\big\}=:\tilde{C}_{p,q}^{(\alpha)}(K).
$$
\end{proposition}

\begin{proof} Since
\begin{eqnarray*}
\|\mu\|_1&=&\mu(K)\\
&\leq& \iint_{\mathbb R^{1+n}_+}S_\alpha F\,d\mu\\
&=&\iint_{\mathbb R^{1+n}_+}F\,S_\alpha^{\ast}\mu\, dxdt\\
&\leq&\|F\|_{L^{q}_{t}L^{p}_{x}(\mathbb R^{1+n}_+)}
\|S_\alpha^{\ast}\mu\|_{L^{q'}_{t}L^{p'}_{x}(\mathbb R^{1+n}_+)},
\end{eqnarray*}
one has
\begin{equation*}
\tilde{C}_{p,q}^{(\alpha)}(K)\leq C_{p,q}^{(\alpha)}(K)
\end{equation*}
for any compact set $K\subset \mathbb R^{1+n}_+$. Moreover, this last inequality is actually an equality - in fact, if
$$
\begin{cases}
X=\{\mu:\ \ \mu\in\mathcal{M}_+(K)\ \&\ \mu(K)=1\};\\
Y=\Big\{F:\ \ 0\le F\in L^q_tL^p_x(\mathbb R^{1+n}_+)\ \&\ \|F\|_{L^q_tL^p_x(\mathbb R^{1+n}_+)}\le 1\Big\};\\
Z=\Big\{F:\ \ 0\le F\in L^q_tL^p_x(\mathbb R^{1+n}_+)\ \&\ S_\alpha F\ge 1_K\Big\};\\
\mathsf{E}(\mu,F)=\iint_{\mathbb R^{1+n}_+}(S_\alpha^\ast\mu)F\,dxdt=\iint_{\mathbb R^{1+n}_+}S_\alpha F\,d\mu,
\end{cases}
$$
then an easy computation, along with an application of \cite[Theorem 2.4.1]{AH}, gives
\begin{eqnarray*}
\min_{\mu\in\mathcal{M}_+(K)}
\frac{\|S^\ast_\alpha\mu\|_{L^{q'}_tL^{p'}_x(\mathbb R^{1+n}_+)}}{\mu(K)}&=&
\min_{\mu\in X}\sup_{F\in Y}\mathsf{E}(\mu,F)\\
&=&\sup_{F\in Y}\min_{\mu\in X}\mathsf{E}(\mu,F)\\
&=&\sup_{0\le F\in L^q_tL^p_x(\mathbb R^{1+n}_+)}\frac{\min_{(t,x)\in K}S_\alpha F(t,x)}{\|F\|_{L^{q}_tL^{p}_x(\mathbb R^{1+n}_+)}}\\
&=&\sup_{F\in Z}\|F\|^{-1}_{L^q_tL^p_x(\mathbb R^{1+n}_+)}\\
&=&\big(C_{p,q}^{(\alpha)}(K)\big)^{-\frac1{p\wedge q}},
\end{eqnarray*}
and hence
\begin{equation*}
\tilde{C}_{p,q}^{(\alpha)}(K)\geq C_{p,q}^{(\alpha)}(K),
\end{equation*}
thereby the desired equality follows.
\end{proof}

\subsection{Essentialness of the $(\alpha,p,q)$-capacity} Some fundamental
properties of the $(\alpha,p,q)$-capacity are stated in the following proposition.

\begin{proposition}\label{p23}

\item{\rm(i)} $C_{p,q}^{(\alpha)}(\emptyset)=0$. Moreover, under $\emptyset\not=K\subset\mathbb R^{1+n}_+$, $C^{(\alpha)}_{p,q}(K)=0$ if and only if there exists $0\le F\in L^q_tL^p_x(\mathbb R^{1+n}_+)$ such that
$$
K\subseteq\{(t,x)\in\mathbb R^{1+n}_+:\ S_\alpha F(t,x)=\infty\}.
$$
\item{\rm(ii)} $K_1\subseteq K_2\subset\mathbb R^{1+n}_+\Longrightarrow C^{(\alpha)}_{p,q}(K_1)\le C^{(\alpha)}_{p,q}(K_2)$.

\item{\rm (iii)}
$$
C^{(\alpha)}_{p,q}\Big(\bigcup_{j=1}^\infty K_{j}\Big)\leq \sum_{j=1}^\infty C^{(\alpha)}_{p,q}(K_{j})
$$
for any sequence $\{K_{j}\}_{j=1}^\infty$ of subsets of $\mathbb{R}^{n+1}_{+}.$

\item{\rm (iv)} $C^{(\alpha)}_{p,q}\big(K+(0,x_{0})\big)=C^{(\alpha)}_{p,q}(K)$ for any $K\subset \mathbb{R}^{n+1}_{+}$ and any $x_{0}\in \mathbb{R}^{n}.$
\end{proposition}

\begin{proof} (i) Only the `iff' part needs an argument. To do so, note that for $\lambda>0$ the inequality
$$
C_{p,q}^{(\alpha)}\big(\{(t,x)\in\mathbb R^{1+n}_+:\ F\ge 0\ \&\ S_\alpha F(t,x)\ge\lambda\}\big)\le \lambda^{-p\wedge q}\|F\|^{p\wedge q}_{L^q_tL^p(\mathbb R^{1+n}_+)}
$$
follows from the definition of $C_{p,q}^{(\alpha)}$. Clearly, this implies
$$
C_{p,q}^{(\alpha)}\big(\mathcal{B}[S_\alpha F;p,q]\big)=0.
$$
Therefore, if $0\le F\in L^q_tL^p_x(\mathbb R^{1+n}_+)$ enjoys $K\subseteq \mathcal{B}[S_\alpha F;p,q],$
then $C_{p,q}^{(\alpha)}(K)=0$ follows from (ii) - the monotonicity of capacity.

Conversely, if $C_{p,q}^{(\alpha)}(K)=0$ then taking nonnegative functions $F_j$ such that
$$
\begin{cases}
S_\alpha F_j(t,x)\ge 1,\ \ \forall (t,x)\in K\\
\hbox{and}\\
\|F_j\|_{L^q_tL^p(\mathbb R^{1+n}_+)}<2^{-j}
\end{cases}
$$
derives that $F=\sum_{j=1}^\infty F_j$ enjoys the required properties.

(ii) This follows from the definition of $(\alpha,p,q)$-capacity.

(iii) The forthcoming argument is standard; see also \cite{Ad, Xiao, AdXi}.

{\it Case 1: $p\geq q$.}\ If we choose $F_{j}$ with $S_\alpha F_{j}\geq 1$ on $K_{j},$ then
$F=\sup_{j=1,2,3,...} F_{j}$ satisfies $S_\alpha F\geq 1$ on $\bigcup_{j=1}^\infty K_{j}$
   and
$$
\|F\|_{L^{q}_{t}L^{p}_{x}(\mathbb{R}^{1+n}_{+})}^{q}
\leq\int_{0}^{\infty}\left(\sum_{j=1}^\infty\int_{\mathbb{R}^{n}}|F_{j}|^{p}dx\right)^{\frac{q}{p}}\,dt
\leq\sum_{j=1}^\infty\int_{0}^{\infty}\left(\int_{\mathbb{R}^{n}}|F_{j}|^{p}dx\right)^{\frac{q}{p}}dt.
$$
So, the desired inequality follows.

{\it Case 2: $p<q$.}\ Now, the Minkowski inequality implies that
\begin{align*}
\|F\|_{L^{q}_{t}L^{p}_{x}(\mathbb{R}^{1+n}_{+})}^{p}&\leq
\left[\int_{0}^{\infty}\left(\sum_{j=1}^\infty\int_{\mathbb{R}^{n}}|F_{j}|^{p}dx\right)^{\frac{q}{p}}dt\right]^{\frac{p}{q}}\\
&\leq \sum_{j=1}^\infty\left[\int_{0}^{\infty}\left(\int_{\mathbb{R}^{n}}|F_{j}|^{p}dx\right)^{\frac{q}{p}}dt\right]^{\frac{p}{q}},
\end{align*}
whence deducing the desired inequality.

(iv) This is a consequence of the following implication:
$$
F_{x_{0}}(t,x)=F(t,x+x_{0})\Longrightarrow\|F_{x_{0}}\|_{L^{q}_{t}L^{p}_{x}(\mathbb{R}^{1+n}_{+})}
=\|F\|_{L^{q}_{t}L^{p}_{x}(\mathbb{R}^{1+n}_{+})},
$$
which completes the proof of Proposition \ref{p23}.
\end{proof}

\section{Proofs of Theorems \ref{t1new}-\ref{t12new} and Corollary \ref{c31}}\label{s3}

Now, we are ready to carry out the task as just mentioned in the title of Section \ref{s3}.

\subsection{Proof of Theorem \ref{t1new}} (i) Let
$$
\begin{cases}
(t,x)\in\mathbb R^{1+n}_+;\\
(t_0,x_0)\in\mathbb R^{1+n}_+;\\
f\in L^p(\mathbb R^n);\\
p\in [1,\infty];\\
0\le t_1<t_2<\infty.
\end{cases}
$$
Since $K^{(\alpha)}_{t_0}(\cdot)$ is of $C^\infty(\mathbb R^n)$, one has that $R_\alpha f(t_0,x)=e^{-t_0(-\Delta)^\alpha}f(x)$ is of $C^\infty(\mathbb R^n)$ too. Meanwhile, for $x\in\mathbb R^n$ one gets
$$
R_\alpha f(t_1,x)-R_\alpha f(t_2,x)=\int_{t_1}^{t_2}(-\Delta)^\alpha e^{-t(-\Delta)^\alpha}f(x)\,dt.
$$
Note that the kernel $\tilde{K}^{(\alpha)}_t(\cdot)$ of $(-\Delta)^\alpha e^{-t(-\Delta)^\alpha}$ obeys
$|\tilde{K}^{(\alpha)}_t(x)|\lesssim {(t^{\frac 1{2\az}}+|x|)^{-n-2\alpha}};$
see also \cite[Lemma 2.2 \& (2.5)]{MiaoYuanZhang}. So, an application of \cite[Lemma 3.1]{MiaoYuanZhang} gives
$$
\Big\|(-\Delta)^\alpha e^{-t(-\Delta)^\alpha}f\Big\|_{L^\infty(\mathbb R^n)}
\lesssim
\begin{cases}
   t^{-1-\frac{n}{2\alpha p}}\|f\|_{L^p(\mathbb R^n)}\ \ \hbox{for}\ \ p\in [1,\infty);\\
    t^{-1}\|f\|_{L^\infty(\mathbb R^n)}\ \ \hbox{for}\ \ p=\infty,
\end{cases}
$$
and hence
$$
|R_\alpha f(t_1,x)-R_\alpha f(t_2,x)|\lesssim
\|f\|_{L^p(\mathbb R^n)}
\begin{cases}\big|t_1^{-\frac{n}{2\alpha p}}-t_2^{-\frac{n}{2\alpha p}}\big|\ \ \hbox{for}\ \ p\in [1,\infty);\\
|\ln t_1-\ln t_2|\ \ \hbox{for}\ \ p=\infty.
\end{cases}
$$
Putting the above facts together yields
\begin{align*}
|R_\alpha f(t,x)-R_\alpha f(t_0,x_0)|&\le|R_\alpha f(t_0,x)-R_\alpha f(t_0,x_0)|+|R_\alpha f(t,x)-R_\alpha f(t_0,x)|\\
&\to 0\ \ \hbox{as}\ \ (t,x)\to (t_0,x_0).
\end{align*}
Therefore $R_\alpha f$ is of $C(\mathbb R^{1+n}_+)$.

(ii) Let $(t,x)\in \rr^{1+n}_+$ be fixed.
Then we have
$$
    |S_\az F(t,x)|\le\int_0^t \int_\rn K_{t-s}^{(\az)}(x-y)|F(s,y)|\,dy\,ds=\mathrm{I}+\mathrm{II},
    $$
    where
    $$
    \begin{cases}
    \mathrm{I}:=\int_0^r\int_\rn K_{t-s}^{(\az)}(x-y)|F(s,y)|\,dy\,ds; \\
    \mathrm{II}:=\int_r^t\int_\rn K_{t-s}^{(\az)}(x-y)|F(s,y)|\,dy\,ds.
\end{cases}
$$
From the H\"older inequality and the assumption
$\frac np+\frac {2\az}q=2\az$ it follows that
\begin{eqnarray*}
\mathrm{I}&&\lesssim\int_0^r\int_{\rn}
   \frac{|t-s|}{(|t-s|^{\frac {1}{2\az}}+|x-y|)^{n+2\alpha}}|F(s,y)|\,dy\,ds\\
   &&\lesssim\int_0^r|t-s| \|F(s,\cdot)\|_{L^p_x(\rn)}\lf(\int_{\rn}
   \frac{dy}{(|t-s|^{\frac {1}{2\az}}+|x-y|)^{(n+2\alpha)(\frac{p}{p-1})}}\r)^{\frac{p-1}{p}}\,ds\\
&&\lesssim\int_0^r\frac{\|F(s,\cdot)\|_{L^p_x(\rn)}}{|t-s|^{\frac{n}{2p\az}}}\,ds\\
&&\lesssim\|F\|_{L^q_tL^p_x(\rr^{1+n}_+)}
\lf(\int_0^r\frac{ds}{|t-s|^{\frac{n}{2p\az}(\frac{q}{q-1})}}\r)^{\frac{q-1}{q}}\\
&&\lesssim\|F\|_{L^q_tL^p_x(\rr^{1+n}_+)}\lf(\ln \frac{t}{t-r}\r)^{\frac{q-1}{q}}.
\end{eqnarray*}
Similarly, by using $\mathrm{M}_\mathbb R$ - the Hardy-Littlewood maximal function on $\mathbb R$, we obtain
\begin{eqnarray*}
\mathrm{II}&&\lesssim\int_r^t\int_{\rn}
   \frac{|t-s|}{(|t-s|^{\frac {1}{2\az}}+|x-y|)^{n+2\alpha}}|F(s,y)|\,dy\,ds\\
&&\lesssim\int_r^t\frac{\|F(s,\cdot)\|_{L^p_x(\rn)}}{|t-s|^{\frac{n}{2p\az}}}\,ds\\
&&\lesssim\sum_{k=0}^{-\fz}\int_{t-2^k|t-r|}^{t-2^{k-1}|t-r|}
\frac{\|F(s,\cdot)\|_{L^p_x(\rn)}}{|t-s|^{\frac{n}{2p\az}}}\,ds\\
&&\lesssim\sum_{k=0}^{-\fz}\frac{1}{(2^k|t-r|)^{\frac{n}{2p\az}}} \int_{t-2^k|t-r|}^{t}
\|F(s,\cdot)\|_{L^p_x(\rn)}\,ds\\
&&\lesssim\sum_{k=0}^{-\fz}(2^k|t-r|)^{1-\frac{n}{2p\az}}\cm_\rr(\|F(\cdot,\cdot)\|_{L^p_x(\rn)})(t) \\
&&\lesssim|t-r|^{1/q}\cm_\rr(\|F(\cdot,\cdot)\|_{L^p_x(\rn)})(t).
\end{eqnarray*}
Via choosing $r\in (0,t)$ such that
$$|t-r|^{1/q}=\min\lf\{t^{1/q}, \,
\frac{\|F\|_{L^q_tL^p_x(\rr^{1+n}_+)}}{\cm_\rr(\|F(\cdot,\cdot)\|_{L^p_x(\rn)})(t)}\r\},
$$
we see that
\begin{eqnarray*}
    |S_\az F(t,x)|\lesssim\|F\|_{L^q_tL^p_x(\rr^{1+n}_+)}
    \max\lf\{1,\,\lf[\ln\frac{et^{1/q}\cm_{\rr} (\|F\|_{L^p_x(\rr^n)})(t)}
{\|F\|_{L^q_tL^p_x(\rr^{1+n}_+)}}\r]^{\frac{q-1}{q}}\r\}.
\end{eqnarray*}
Letting $r_0=t_0^{\frac{1}{2\az}}$ yields a constant $C>0$ such that
\begin{eqnarray*}
   &&\iint_{B_{r_0}^{(\alpha)}{(t_0,x_0)}}\exp\lf(\frac{S_\az F(t,x)}
   {C\|F\|_{L^q_tL^p_x(\rr^{1+n}_+)}}\r)^{\frac{q}{q-1}}\,dx\,dt\\
   &&\hs\lesssim \iint_{B_{r_0}^{(\alpha)}{(t_0,x_0)}}\frac{et^{1/q}\cm_{\rr}
   (\|F\|_{L^p_x(\rr^n)})(t)}
{\|F\|_{L^q_tL^p_x(\rr^{1+n}_+)}}\,dx\,dt\\
&&\hs\lesssim r_0^nt_0^{1/q}\int_{0}^{2t_0}
\frac{\cm_{\rr} (\|F\|_{L^p_x(\rr^n)})(t)}
{\|F\|_{L^q_tL^p_x(\rr^{1+n}_+)}}\,dt\\
&&\hs\lesssim t_0^{1/q}r_0^{n+2\az-2\az/q}\\
&&\hs\approx |B_{r_0}^{(\alpha)}{(t_0,x_0)}|,
\end{eqnarray*}
which completes the proof of (ii).

(iii) Given a point $(t_0,x_0)\in \rr^{1+n}_+$, let $x\in\rn$ be sufficient
close to $x_0$ and $\dz=|x-x_0|$. Then
\begin{eqnarray*}
   &&|S_\az F(t_0,x_0)- S_\az F(t_0,x)|\\
&&\hs \le \int_0^{t_0}\int_\rn|K^{(\az)}_{t_0-s}(x_0-y)
   -K^{(\az)}_{t_0-s}(x-y)||F(y,s)|\,dy\,ds\\
   &&\hs \le \int_0^{t_0}\int_{B(x_0,3\dz)}\cdots\,dyds+\int_0^{t_0}\int_{\rn\setminus B(x_0,3\dz)}\cdots\,dyds\\
   &&\hs=:\mathrm{I}+\mathrm{II}.
\end{eqnarray*}
Note that
\begin{eqnarray*}
&&\int_0^{t_0}\int_{B(x_0,3\dz)}K^{(\az)}_{t_0-s}(x_0-y)||F(y,s)|\,dy\,ds\\
&&\hs\le \int_{0}^{t_0-(2\dz)^{2\az}}\int_{B(x_0,3\dz)}\Big(\frac{|t-s|}
{|t-s|^{1+\frac{n}{2\az}}}\Big)|F(y,s)|\,dy\,ds\\
&&\hs\hs+\int_{t_0-(2\dz)^{2\az}}^t\int_{B(x_0,3\dz)}\Big(\frac{|t-s|}
{(|t-s|^{\frac 1{2\az}}+|x_0-y|)^{n+\az}}\Big)|F(y,s)|\,dy\,ds\\
&&\hs \lesssim\int_{0}^{t_0-(2\dz)^{2\az}}\Big(\frac{\dz^{\frac {n(p-1)}{p}}}
{|t-s|^{\frac{n}{2\az}}}\Big)\|F(\cdot,s)\|_{L^p_x(\rn)}\,ds\\
&&\hs \hs+\ \int_{t_0-(2\dz)^{2\az}}^t\frac{\|F(\cdot,s)\|_{L^p_x(\rn)}}{|t-s|^{\frac{n}{2p\az}}}
\,ds\\
&&\hs \lesssim\|F\|_{L^q_tL^p_x(\rr^{1+n}_+)}\dz^{\frac {n(p-1)}{p}}\lf(\int_{0}^{t_0-(2\dz)^{2\az}} \frac{ds}
{|t-s|^{\frac{nq}{2\az(q-1)}}}\r)^{\frac{q-1}{q}}\\
&&\hs \hs+\ \|F\|_{L^q_tL^p_x(\rr^{1+n}_+)}
\lf(\int_{t_0-(2\dz)^{2\az}}^t\frac{ds}{|t-s|^{\frac{nq}{2p\az(q-1)}}}\r)^{\frac {q-1}{q}}\\
&&\hs\lesssim\|F\|_{L^q_tL^p_x(\rr^{1+n}_+)}\lf(\dz^{\frac {n(p-1)}{p}}\dz^{\frac{2\az(q-1)}{q}-n}
+\ \dz^{\frac{2\az(q-1)}{q}-\frac{n}{p}}\r)\\
&&\hs\lesssim\|F\|_{L^q_tL^p_x(\rr^{1+n}_+)}\dz^{\frac{2\az(q-1)}{q}-\frac{n}{p}}.
\end{eqnarray*}
Thus the first term $\mathrm{I}$ is bounded from above as
\begin{eqnarray*}
\mathrm{I} && \le \int_0^{t_0}\int_{B(x_0,3\dz)}K^{(\az)}_{t_0-s}(x_0-y)||F(y,s)|\,dy\,ds\\
&&\hs+ \int_0^{t_0}\int_{B(x,4\dz)}K^{(\az)}_{t_0-s}(x-y)||F(y,s)|\,dy\,ds\\
   &&\lesssim\|F\|_{L^q_tL^p_x(\rr^{1+n}_+)}|x-x_0|^{\frac{2\az(q-1)}{q}-\frac{n}{p}}.
\end{eqnarray*}
To estimate the second term $\mathrm{II}$, notice that
$$|\nabla K^{(\az)}_1(x)|\lesssim(1+|x|)^{-n-1};
$$
see also \cite[Remark 2.1]{MiaoYuanZhang}. Using this and the H\"older inequality, we have
\begin{eqnarray*}
\mathrm{II} &&
\le \int_0^{t_0}\int_{\rn\setminus B(x_0,3\dz)}|K^{(\az)}_{t_0-s}(x_0-y)
   -K^{(\az)}_{t_0-s}(x-y)||F(y,s)|\,dy\,ds\\
   &&\le \int_0^{t_0}\int_{\rn\setminus B(x_0,3\dz)}\Big(\frac{\dz}{|t-s|^{\frac{1}{2\az}}}\Big)
\Big(\frac{|t-s|^{\frac{1}{2\az}}}{(|t-s|^{\frac{1}{2\az}}+|x_0-y|)^{n+1}}\Big)|F(y,s)|\,dy\,ds\\
&&\lesssim\int_{0}^{t_0-(2\dz)^{2\az}}\int_{\rn\setminus B(x_0,3\dz)}\Big(\frac{\dz}
{(|t-s|^{\frac 1{2\az}}+|x_0-y|)^{n+1}}\Big)|F(y,s)|\,dy\,ds\\
&&\hs\hs+\int_{t_0-(2\dz)^{2\az}}^t\int_{\rn\setminus B(x_0,3\dz)}\Big(\frac{\dz}
{|x_0-y|^{n+1}}\Big)|F(y,s)|\,dy\,ds\\
&&\lesssim\int_{0}^{t_0-(2\dz)^{2\az}}\dz \|F(\cdot,s)\|_{L^p_x(\rn)}
|t-s|^{-\frac{n}{2\az p}-\frac{1}{2\az}}\,ds\\
&&\hs\hs+\ \int_{t_0-(2\dz)^{2\az}}^t\|F(\cdot,s)\|_{L^p_x(\rn)}\dz^{-\frac {n}{p}}\,ds\\
&&\lesssim \|F\|_{L^q_tL^p_x(\rr^{1+n}_+)}|x-x_0|^{\frac{2\az(q-1)}{q}-\frac{n}{p}}.
\end{eqnarray*}
Thus, we conclude that
\begin{eqnarray*}
   &&|S_\az F(t_0,x_0)- S_\az F(t_0,x)|\lesssim
   \|F\|_{L^q_tL^p_x(\rr^{1+n}_+)}|x-x_0|^{\frac{2\az(q-1)}{q}-\frac{n}{p}}.
\end{eqnarray*}

Let $(x,t_1), (x,t_2)\in  \rr^{1+n}_+$. Without loss of generality we may assume $t_1>t_2$, and then write
\begin{eqnarray*}
   |S_\az F(t_1,x)-S_\az F(t_2,x)|&&\le \int_0^{t_2}
   \lf|\big(e^{-(t_1-s)(-\Delta)^\az}-e^{-(t_2-s)(-\Delta)^\az}\big)F(x,s)\r|\,ds\\
   &&\hs+\int_{t_2}^{t_1}
   \lf|(e^{-(t_1-s)(-\Delta)^\az})F(x,s)\r|\,ds\\
   &&=:\mathrm{III}+\mathrm{IV}.
\end{eqnarray*}
By using the mapping property of the semigroup, we obtain
\begin{eqnarray*}
   \mathrm{III}&&\le \int_0^{t_2}\int_{t_2-s}^{t_1-s}|(-\Delta)^\az
   e^{-r(-\Delta)^\az}F(x,s)|\,dr\,ds\\
   &&\le \int_0^{t_2}\int_{t_2-s}^{t_1-s}r^{-1-\frac{n}{2\az p}}\|F(\cdot,s)\|_{L^p_x(\rn)}\,dr\,ds\\
   &&\le \int_0^{t_2}\int_0^{t_1-t_2} (t_2-s+r)^{-1-\frac{n}{2\az p}}
   \|F(\cdot,s)\|_{L^p_x(\rn)}\,ds\\
   &&\le \int_0^{t_1-t_2}\int_0^{t_2} (t_2-s+r)^{-1-\frac{n}{2\az p}}\|F(\cdot,s)\|_{L^p_x(\rn)}\,ds\\
   &&\lesssim\|F\|_{L^q_tL^p_x(\rr^{1+n}_+)}\int_0^{t_1-t_2} r^{\frac {q-1}{q}-1-\frac{n}{2\az p}}\,dr\\
   &&\lesssim|t_2-t_1|^{1-\frac{1}{q}-\frac{n}{2\az p}}\|F\|_{L^q_tL^p_x(\rr^{1+n}_+)},
\end{eqnarray*}
and
\begin{eqnarray*}
   \mathrm{IV}&&\le \int_{t_2}^{t_1}(t_1-s)^{-\frac{n}{2\az p}}\|F(s,\cdot)\|_{L^p_x(\rn)}\,ds\\
   &&\lesssim|t_2-t_1|^{1-\frac{1}{q}-\frac{n}{2\az p}}\|F\|_{L^q_tL^p_x(\rr^{1+n}_+)}.
\end{eqnarray*}
Hence
\begin{eqnarray*}
   |S_\az F(t_1,x)-S_\az F(t_2,x)|&&\lesssim|t_2-t_1|^{1-\frac{1}{q}-\frac{n}{2\az p}}
   \|F\|_{L^q_tL^p_x(\rr^{1+n}_+)}
\end{eqnarray*}
The difference estimates on $S_\az$ give us that if $(t,x)$ is close to
$(t_0,x_0)$ then
\begin{eqnarray*}
   &&|S_\az F(t,x)-S_\az F(t_0,x_0)|\\
   &&\ \ \le |S_\az F(t,x)-S_\az F(t_0,x)|+|S_\az F(t_0,x)-S_\az F(t_0,x_0)|\\
   &&\ \ \lesssim\lf(|t-t_0|^{1-\frac{1}{q}-\frac{n}{2\az p}}+|x-x_0|^{\frac{2\az(q-1)}{q}-\frac{n}{p}}\r)
   \|F\|_{L^q_tL^p_x(\rr^{1+n}_+)},
\end{eqnarray*}
which completes the proof of (iii).

\subsection{Proof of Theorem \ref{t12new}} (i) In the sequel, let $\beta=(p\wedge q)(\frac{n}{p}+\frac{2\alpha}{q}-2\alpha)$. Also, assume that $F\geq 0$ satisfies $S_\alpha F\geq 1_{B^{(\alpha)}_{r_0}(0,0)}.$
Then, according to the definition of operator $S_\alpha$, the following transform
$$
\begin{cases}
s=\frac{t}{r_0^{2\alpha}};\\
y=\frac{x}{r_0};\\
F_{r_0}(s,y)=F(r_0^{2\alpha}s,r_0y);\\
G(s,y)=r_0^{2\alpha}F_{r_0}(s,y),
\end{cases}
$$
enjoys the property $S_\alpha G\geq 1_{B_{1}^{(\alpha)}(0,0)}$. Thus,
$$
C_{p,q}^{(\alpha)}(B^{(\alpha)}_{1}(0,0))\leq\|r_0^{2\alpha}F_{r_0}\|_{L^{q}_{t}L^{p}_{x}(\mathbb{R}^{1+n}_{+})}^{p\wedge q}=r_0^{-\beta}\|F\|_{L^{q}_{t}L^{p}_{x}(\mathbb{R}^{1+n}_{+})}^{p\wedge
q}.
$$
This implies that
$$
C_{p,q}^{(\alpha)}(B_{1}^{(\alpha)}(0,0))\leq r_0^{-\beta}
C_{p,q}(B_{r_0}^{(\alpha)}(0,0)).
$$
In fact, the last inequality is an equality since changing the order of $B^{(\alpha)}_{1}(0,0)$ and $B^{(\alpha)}_{r_0}(0,0)$
derives
$$
C_{p,q}^{(\alpha)}(B^{(\alpha)}_{r_0}(0,0))\leq r_0^{\beta}
C_{p,q}^{(\alpha)}(B^{(\alpha)}_{1}(0,0)).
$$

Next, we consider the desired equivalence estimate. If $F\geq 0$ and $S_\alpha F\geq
1_{B^{(\alpha)}_{r_0}(t_{0},x_{0})}$, then, for $1\leq p<\infty$ and
$1<q<\infty,$ there exist $\tilde{p}$ and $\tilde{q}$ such that
$$
\begin{cases}
1\leq p <\tilde{p}<\infty;\\
1<q<\tilde{q}<\infty;\\
\left(\frac{1}{q}-\frac{1}{\tilde{q}}\right)
+\frac{n}{2\alpha}\left(\frac{1}{p}-\frac{1}{\tilde{p}}\right)=1.
\end{cases}
$$
Consequently, according to Theorem \ref{t11}(ii) we have
\begin{equation*}\label{ineq str}
\|S_\alpha F\|_{L^{\tilde{q}}_tL^{\tilde{p}}_{x}(\mathbb{R}^{1+n}_{+})}\lesssim
\|F\|_{L^{q}_{t}L^{p}_{x}(\mathbb{R}^{1+n}_{+})}.
\end{equation*}
This, along with the definition of $C^{(\alpha)}_{p,q}(\cdot)$, implies that
$$
r_0^{\beta}\lesssim C^{(\alpha)}_{p,q}\big(B^{(\alpha)}_{r_0}(t_{0},x_{0})\big)
$$
thanks to
$$
\frac{n}{\tilde{p}}+\frac{2\alpha}{\tilde{q}}=\frac{n}{p}+\frac{2\alpha}{q}-2\alpha.
$$

To get the corresponding upper bound of $C^{(\alpha)}_{p,q}\big(B^{(\alpha)}_{r_0}(t_{0},x_{0})\big)$, we consider
$$
B^{(\alpha)}_{r_0,\eta}(t_{0},x_{0}):=\{(t,x)\in\mathbb{R}^{1+n}_{+}:\
|t-t_{0}|<(\eta r_0)^{2\alpha}\ \&\ |x-x_0|<r_0\}
$$
for some sufficiently large $\eta>0$ which will be determined later. Note that $(t,x)\in B^{(\alpha)}_{r_0}(t_{0},x_{0})$ ensures
\begin{eqnarray*}
&&S_\alpha 1_{B^{(\alpha)}_{r_0,\eta}(t_{0},x_{0})}(t,x)\\
&&\hs=\int_{0}^{t}\int_{\mathbb{R}^{n}}K_{t-s}^{(\alpha)}(x-y)1_{B^{(\alpha)}_{r_0,\eta}(t_{0},x_{0})}(s,y)\,dyds\\
&&\hs=\int_{(0,t)\cap\{s:|s-t_{0}|<(\eta r_0)^{2\alpha}\}}\int_{|y-x_{0}|<r_0}K_{t-s}^{(\alpha)}(x-y)\,dyds\\
&&\hs\geq\int_{(0,t)\cap \{s:|s-t_{0}|<(\eta r_0)^{2\alpha}\}
\cap\{s:t-s>\frac{\eta^{2\alpha}-1}{2}r_0^{2\alpha}\}}\int_{|y-x_{0}|<r_0}K_{t-s}^{(\alpha)}(x-y)\,dyds
\end{eqnarray*}
for sufficiently small $r_0>0.$ According to \cite[Proposition 1]{M.
Nishio}, there are positive constants $\sigma$ and $\kappa$, depending only $n$ and $\alpha$, such that
$$
\inf\{K_{t}^{(\alpha)}(x): |x|\leq \sigma
t^{\frac{1}{2\alpha}}\}\geq \kappa t^{-\frac{n}{2\alpha}}.
$$
Under
$$
\begin{cases}
(t,x)\in B^{(\alpha)}_{r_0}(t_{0},x_{0});\\
|y-x_{0}|<r_0;\\
t-s>\Big(\frac{\eta^{2\alpha}-1}{2}\Big)r_0^{2\alpha},
\end{cases}
$$
one has
$$
|x-y|\leq
|x-x_{0}|+|y-x_{0}|<2r_0<2\left(\frac{2}{\eta^{2\alpha}-1}\right)^{\frac{1}{2\alpha}}
|t-s|^{\frac{1}{2\alpha}}<\sigma|t-s|^{\frac{1}{2\alpha}}
$$
for some large enough $\eta$ with
$$
2\Big(\frac{2}{\eta^{2\alpha}-1}\Big)^{\frac{1}{2\alpha}}<\sigma.
$$
Thus, one gets that if $|t-t_0|<r^{2\alpha}$ then
\begin{align*}
&S_\alpha 1_{B^{(\alpha)}_{r_0,\eta}(t_{0},x_{0})}(t,x)\\
&\quad\geq\int_{(0,t)\cap
\{s:|s-t_{0}|<(\eta
r_0)^{2\alpha}\}\cap\{s:t-s>{(\eta^{2\alpha}-1)}{2^{-1}}r_0^{2\alpha}\}}\int_{|y-x_{0}|<
r_0} |t-s|^{-\frac{n}{2\alpha}}\,dyds\\
&\quad\ge cr_0^{2\alpha}
\end{align*}
holds for some constant $c>0$ independent of $r_0$. Consequently,
$$
S_\alpha\Big(\frac{1_{B^{(\alpha)}_{r_0,\eta}(t_{0},x_{0})}}{cr_0^{2\alpha}}\Big)
(t,x)\geq 1,\quad\forall (t,x)\in B^{(\alpha)}_{r_0}(t_{0},x_{0}).
$$
This gives
$$
C_{p,q}^{(\alpha)}(B^{(\alpha)}_{r_0}(t_{0},x_{0}))\leq
\left\|\frac{1_{B^{(\alpha)}_{r_0,\eta}(t_{0},x_{0})}}{cr_0^{2\alpha}}
\right\|_{L^{q}_{t}L^{p}_{x}(\mathbb{R}^{1+n}_{+})}^{p\wedge q}\lesssim r_0^{\beta}.
$$

(ii) For an arbitrarily fixed point $(t_0,x_0)\in \rr^{1+n}_+$. Let $r_0<<\min\{t_0,1\}$.
Suppose that $S_\alpha F(t,x)\ge 1$ on $B_{r_0}^{(\alpha)}(t_0,x_0)$.
Then by Theorem \ref{t1new}(ii), we have a constant $C>0$ such that
\begin{eqnarray*}
&&\iint_{B_{r_0}^{(\alpha)}{(t_0,x_0)}}\exp\lf(\frac{S_\az F(t,x)}
   {C\|F\|_{L^q_tL^p_x(\rr^{1+n}_+)}}\r)^{\frac{q}{q-1}}\,dx\,dt\\
&&\hs\lesssim \iint_{B_{r_0}^{(\alpha)}{(t_0,x_0)}}\frac{et^{1/q}\cm_{\rr}
   (\|F\|_{L^p_x(\rr^n)})(t)}
{\|F\|_{L^q_tL^p_x(\rr^{1+n}_+)}}\,dx\,dt\\
&&\hs\lesssim r_0^{n+2\az}+r_0^nt_0^{1/q}\int_{t_0-r_0^{2\az}}^{t_0+r_0^{2\az}}
\frac{\cm_{\rr} (\|F\|_{L^p_x(\rr^n)})(t)}
{\|F\|_{L^q_tL^p_x(\rr^{1+n}_+)}}\,dt\\
&&\hs\lesssim t_0^{1/q}r_0^{n+2\az-2\az/q}.
\end{eqnarray*}

On the other hand, as $S_\alpha F(\cdot,\cdot)\ge 1$ on $B_{r_0}^{(\alpha)}(t_0,x_0)$,
it follows that for a constant $c>0$,
$$
   \iint_{B_{r_0}^{(\alpha)}{(t_0,x_0)}}\exp\lf(\frac{S_\az F(t,x)}
   {c[\ln \frac 1{r_0}]^{\frac{1-q}{q}}}\r)^{\frac{q}{q-1}}\,dx\,dt
   \gtrsim r_0^{n+2\az}\exp\lf( c^{-\frac{q}{q-1}}\ln \frac 1{r_0}\r)\gtrsim r_0^{n+2\az-c^{-\frac{q}{q-1}}},
$$
which implies that
$$\|F\|_{L^q_tL^p_x(\rr^{1+n}_+)}\gtrsim \Big(\ln \frac 1{r_0}\Big)^{\frac{1-q}{q}}$$
and hence
$$C_{p,q}^{(\alpha)}(B^{(\alpha)}_{r_0}(t_{0},x_{0}))\gtrsim\Big(\ln \frac 1{r_0}\Big)^{\frac{1-q}{q}(p\wedge q)}\quad\hbox{as}\quad
r_0\to 0.
$$

Next, we prove the converse form of the last inequality. Let
$$
E:=\big\{(t,x)\in \rr^{1+n}_+:\ (2r_0)^{2\az}<t_0-t<(2r_0)^\az\ \&\ |t-t_0|^{\frac 1{2\az}}<|x_0-x|<2\big\}.
$$
Define
$$F(x,t):=\lf\{
\begin{array}{cc}
\frac{1}{(|t_0-t|^{\frac{1}{2\az}}+|x-x_0|)^{2\az}}, & \ \ {\forall}
(t,x)\in E;\\
0, & \mbox{otherwise}.
\end{array}
\r.
$$
By using the known estimate below
$$K_t^{(\az)}(x-y)\approx \frac{t}{(t^{\frac 1{2\az}}+|x-y|)^{n+2\az}},$$
we see that for each $(t,x)\in B_{r_0}^{(\alpha)}(t_0,x_0)$,
\begin{eqnarray*}
S_\az F(t,x)&&=\int_0^t \int_\rn K^{(\az)}_{t-s}(x-y)F(y,s)\,dy\,ds\\
&&\approx\iint_E \frac{|t-s|}{(|t-s|^{\frac 1{2\az}}+|x-y|)^{n+2\az}}F(y,s)\,dy\,ds\\
&&\approx \iint_E \frac{|t_0-s|}{(|t_0-s|^{\frac 1{2\az}}+|x_0-y|)^{n+2\az}}F(y,s)\,dy\,ds\\
&&\gtrsim\int_{t_0-(2r_0)^{\az}}^{t_0-(2r_0)^{2\az}}
\int_{B(x_0,2)\setminus B(x_0,|t_0-s|^{\frac{1}{2\az}})}
\frac{|t_0-s|}{(|t_0-s|^{\frac 1{2\az}}+|x_0-y|)^{n+4\az}}\,dy\,ds\\
&&\gtrsim\!\!\int_{t_0-(2r_0)^{\az}}^{t_0-(2r_0)^{2\az}}
\int_{|t_0-s|^{\frac{1}{2\az}}}^2
\frac{|t_0-s|r^{n-1}}{(|t_0-s|^{\frac 1{2\az}}+r)^{n+4\az}}\,dr\,ds\\
&&\gtrsim\int_{t_0-(2r_0)^{\az}}^{t_0-(2r_0)^{2\az}}|t_0-s|\lf(|t_0-s|^{-2}-2^{-4\az}\r)\,ds\\
&&\gtrsim \int_{t_0-(2r_0)^{\az}}^{t_0-(2r_0)^{2\az}}|t_0-s|^{-1}\,ds\\
&&\gtrsim\ln \frac{1}{(2r_0)^{\az}}.
\end{eqnarray*}
Moreover, noticing that $2p\az>n$, we have
\begin{eqnarray*}
\|F\|_{L^q_tL^p_x(\rr^{1+n}_+)}^q\!\!&&\lesssim\int_{t_0-(2r_0)^{\az}}^{t_0-(2r_0)^{2\az}}
\lf(\int_{B(x_0,2)\setminus B(x_0,|t_0-s|^{\frac{1}{2\az}})}
\frac{1}{(|t_0-s|^{\frac 1{2\az}}+|x_0-y|)^{2p\az}}\,dy\r)^{q/p}\!\!\,ds\\
&&\lesssim\int_{t_0-(2r_0)^{\az}}^{t_0-(2r_0)^{2\az}}
\lf(\int_{|t_0-s|^{\frac{1}{2\az}}}^2
{r^{n-1-2p\az}}\,dr\r)^{q/p}\,ds\\
&&\lesssim\int_{t_0-(2r_0)^{\az}}^{t_0-(2r_0)^{2\az}}|t_0-s|^{\frac{(n-2p\az)q}{2p\az}}\,ds\\
&&\lesssim\ln \frac{1}{(2r_0)^{\az}}.
\end{eqnarray*}
The above two estimates give
$$
C_{p,q}^{(\alpha)}\big(B^{(\alpha)}_{r_0}(t_{0},x_{0})\big)\lesssim\lf\|\frac{F}{\ln \frac 1{r_0}}\r\|_{L^q_tL^p_x(\rr^{1+n}_+)}^{q\wedge p}
\lesssim\lf(\ln \frac 1{r_0}\r)^{\frac{1-q}{q}(p\wedge q)}\quad\hbox{as}\quad r_0\to 0.
$$

\subsection{Proof of Corollary \ref{c31}} (i) This follows from Theorem \ref{t11}(ii), Theorem \ref{t12new}(i) and Proposition \ref{p23}(iii).

(ii)-(iii) The comparison inequalities for the two capacities follow from Proposition \ref{p23}(iii)
and (i)-(ii) of Theorem \ref{t12new}. To get the dimension inequality, we firstly keep in mind the fact
$$
C_{p,q}^{(\alpha)}\big(\mathcal{B}[S_\alpha F;p,q]\big)=0\ \ \hbox{for}\ \ 0\le F\in L^q_tL^p_x(\mathbb R^{1+n}_+),\\
$$
and secondly recall Proposition \ref{p22} and the following Frostman type theorem (cf. \cite[Theorem 5.1.12]{AH}): if $\phi: [0,\infty)\mapsto [0,\infty]$ increases with $\phi(0)=0$ then for a given compact $K\subset\mathbb R^{1+n}_+$ there is a measure $\mu\in\mathcal{M}^+(K)$ obeying
$\mu\big(B_r^{(\alpha)}(t,x)\big)\lesssim \phi(r)$ such that
$\mu(K)\approx H^{\phi,\alpha}_\infty(K)$.

Now, let $K$ be any compact subset of the blow-up set $\mathcal{B}[S_\alpha F; p,q]$ and be contained in a ball $B^{(\alpha)}_R(t_0,x_0)$. Taking $0\le G\in L^{q}_tL_x^p(\mathbb R^{1+n}_+)$ such that $S_\alpha G\ge 1$ on $K$, we use the dyadic decomposition of a set and the H\"older inequality to get that if $0<R_0<1\wedge R$ then
\begin{align*}
\mu(K)&\le\iint_K S_\alpha G(t,x)\,d\mu(t,x)\\
&\le\iint_{\rr^{1+n}_+}G(s,y)\iint_{K\cap((s,\infty)\times\mathbb R^n)}K^{(\alpha)}_{t-s}(x-y)\,d\mu(t,x)\,dyds\\
&\lesssim\iint_{\rr^{1+n}_+}G(s,y)\iint_{K\cap((s,\infty)\times\mathbb R^n)}\Big(\frac{|t-s|}{(|t-s|^\frac1{2\alpha}+|x-y|)^{n+2\alpha}}\Big)\,d\mu(t,x)\,dyds\\
&\ls \mu(K)\iint_{\rr^{1+n}_+\setminus B_{2R_0}^{(\az)}(t_0,x_0)}G(s,y)
\Big(\frac{|t_0-s|}{(|t_0-s|^\frac1{2\alpha}+|x_0-y|)^{n+2\alpha}}\Big)\,dyds\\
&\hs+ \iint_{B_{2R_0}^{(\az)}(t_0,x_0)}G(s,y)\Big(\sum_{j=0}^\infty \frac{\mu\big(B^{(\alpha)}_{2^{-j}R_0}(s,y)\big)}{(2^{-j}R_0)^n}\Big)\,dyds\\
&\ls \mu(K)R_0^{2\az-\frac{2\az}{q}-\frac{n}{p}}\|G\|_{L^q_tL^p_x(\rr^{1+n}_+)}\\
&\hs+\iint_{B_{2R_0}^{(\az)}(t_0,x_0)}G(s,y)\int_0^{R_0} \frac{\mu\big(B^{(\alpha)}_r(s,y)\big)}{r^{1+n}}\,dr\,dyds\\
&\lesssim \mu(K)R_0^{2\az-\frac{2\az}{q}-\frac{n}{p}}\|G\|_{L^q_tL^p_x(\rr^{1+n}_+)}\\
&\hs +\int_0^{R_0}\iint_{B_{2R_0}^{(\az)}(t_0,x_0)} G(s,y){\mu\big(B^{(\alpha)}_r(s,y)\big)}\,dyds\,\frac{dr}{r^{1+n}}.
\end{align*}

For $p\le q$, we have
\begin{align*}
&\int_0^{R_0}\iint_{B_{2R_0}^{(\az)}(t_0,x_0)} G(s,y){\mu\big(B^{(\alpha)}_r(s,y)\big)}\,dyds\,\frac{dr}{r^{1+n}}\\
&\hs\le \|G\|_{L^p_tL^p_x(B_{2R_0}^{(\az)}(t_0,x_0))} \int_0^{R_0}\lf(\iint_{B_{2R_0}^{(\az)}(t_0,x_0)}
   {\mu\big(B^{(\alpha)}_r(s,y)\big)}^{\frac{p}{p-1}}\,dyds\r)^{\frac{p-1}{p}}\,\frac{dr}{r^{1+n}}\\
   &\hs \ls R_0^{2\az (\frac{1}{p}-\frac 1q)}\|G\|_{L^q_tL^p_x(\rr^{1+n}_+)} \int_0^{R_0}\lf(\iint_{B_{2R_0}^{(\az)}(t_0,x_0)}
   {\mu\big(B^{(\alpha)}_r(s,y)\big)}\,dyds\r)^{\frac{p-1}{p}}\,\frac{\phi(r)^{\frac{1}{p}}dr}{r^{1+n}}\\
   &\hs \ls R_0^{2\az (\frac{1}{p}-\frac 1q)}\|G\|_{L^q_tL^p_x(\rr^{1+n}_+)}\mu(K)^{\frac{p-1}{p}}
   \int_0^{R_0}\,\phi(r)^{\frac{1}{p}} r^{-1+2\az-\frac{n+2\az}{p}}\,dr.
\end{align*}
Meanwhile, for $p>q$, it holds that
\begin{align*}
&\int_0^{R_0}\iint_{B_{2R_0}^{(\az)}(t_0,x_0)} G(s,y){\mu\big(B^{(\alpha)}_r(s,y)\big)}\,dyds\,\frac{dr}{r^{1+n}}\\
&\hs\le \|G\|_{L^q_tL^q_x(B_{2R_0}^{(\az)}(t_0,x_0))} \int_0^{R_0}\lf(\iint_{B_{2R_0}^{(\az)}(t_0,x_0)}
   {\mu\big(B^{(\alpha)}_r(s,y)\big)}^{\frac{q}{q-1}}\,dyds\r)^{\frac{q-1}{q}}\,\frac{dr}{r^{1+n}}\\
   &\hs \ls R_0^{n(\frac{1}{q}-\frac 1p)}\|G\|_{L^q_tL^p_x(\rr^{1+n}_+)} \int_0^{R_0}\lf(\iint_{B_{2R_0}^{(\az)}(t_0,x_0)}
   {\mu\big(B^{(\alpha)}_r(s,y)\big)}\,dyds\r)^{\frac{q-1}{q}}\,\frac{\phi(r)^{\frac{1}{q}}dr}{r^{1+n}}\\
   &\hs \ls R_0^{n(\frac{1}{p}-\frac 1q)}\|G\|_{L^q_tL^p_x(\rr^{1+n}_+)}\mu(K)^{\frac{q-1}{q}}
   \int_0^{R_0}\,\phi(r)^{\frac{1}{q}} r^{-1+2\az-\frac{n+2\az}{q}}\,dr.
\end{align*}

The above estimates induce a constant $c_0:=C(R_0,p,q,\alpha)>0$, depending on $R_0$ and $p,q,\alpha$, such that
$$
\mu(K)\lesssim c_0\|G\|_{L_t^qL_x^p(\mathbb R^{1+n}_+)}\lf(\mu(K)+\mu(K)^{1-\frac{1}{p\wedge q}}
\int_0^{R_0}\phi(r)^{\frac1{p\wedge q}} r^{-1+2\az-\frac{n+2\az}{p\wedge q}}\,dr\r).
$$
Therefore, if
$$
\mathrm{III}:=\int_0^{R_0} \phi(r)^{\frac1{p\wedge q}} r^{-1+2\az-\frac{n+2\az}{p\wedge q}}\,dr<\infty,
$$
then by the fact $C_{p,q}^{(\alpha)}(K)=0$ it follows that $\mu(K)=0$, and hence $H^{\phi,\alpha}_\infty(K)\approx\mu(K)=0$. This in turn implies $H^{\phi,\alpha}(K)=0$ thanks to  $$H^{\phi,\alpha}_\infty(\cdot)=0\Longleftrightarrow H^{\phi,\alpha}(\cdot)=0.
$$
Consequently, $H^{\phi,\alpha}(\mathcal{B}[S_\alpha F;p,q])=0$.

The remaining is to consider two situations as follows.

{\it Case 1}: $n-2\alpha(p\wedge q-1)>0$. Under this condition, we choose
$$
\phi(r):=r^\eta,\ \forall r\in (0,\fz)\ \ \&\ \ \eta>n-2\alpha(p\wedge q-1)
$$
to obtain $\mathrm{III}<\infty$, thereby reaching
$$
\hbox{dim}_H^{(\alpha)}\big(\mathcal{B}[S_\alpha F;p,q]\big)\le n-2\alpha(p\wedge q-1).
$$

{\it Case 2}: $n-2\alpha(p\wedge q-1)=0$. Under this condition, we select
$$
\phi_\epsilon (r):=\Big(\ln_+\frac1r\Big)^{-\eta_\epsilon},\ \forall r\in(0,\fz)\ \ \&\ \
\eta_\epsilon=p\wedge q+\epsilon>p\wedge q
$$
to ensure $\mathrm{III}<\infty$ and thus
$H^{\phi_\epsilon,\alpha}(\mathcal{B}[S_\alpha F;p,q])=0.$

\end{document}